\documentclass[12pt,a4paper]{article}
\usepackage{amsmath,amsthm,amssymb,amsfonts,color,graphicx}

\newcommand{\C}{{\mathbb{C}}}          
\newcommand{\Octoni}{{\mathbb{O}}}     
\newcommand{\Proj}{{\mathbb{P}}}        
\newcommand{\R}{{\mathbb{R}}}          
\newcommand{\Z}{{\mathbb{Z}}}          

\newcommand{\g}{{\mathfrak{g}}}       %
\newcommand{\h}{{\mathfrak{h}}}       %
\newcommand{\m}{{\mathfrak{m}}}       %

\newcommand{\sol}{{\mathfrak{so}}}
\newcommand{\sul}{{\mathfrak{su}}}

\newcommand{\rr}{\rightarrow}
\newcommand{\lrr}{\longrightarrow}

\newcommand{\calR}{{{\cal R}^U}}             %

\newcommand{\ch}{{\mathrm{c}}}
\newcommand{\LC}{{\mathrm{L{\text{-}}C}}}
\newcommand{\na}{{\nabla}}
\newcommand{\nag}{{\nabla^g}}
\newcommand{\nac}{{\nabla^\ch}}

\newcommand{\End}[1]{{\mathrm{End}}\,{#1}}

\newcommand{\Aut}[1]{{\mathrm{Aut}}\,{#1}}

\newcommand{\dx}{{\mathrm{d}}}

\newcommand{\vol}{{\mathrm{vol}}}
\newcommand{\Vol}{{\mathrm{Vol}_{SM}}}
\newcommand{\ric}{{\mathrm{Ric}\,}}

\newtheorem{teo}{Theorem}[section]

\newtheorem{coro}{Corollary}[section]
\newtheorem{prop}{Proposition}[section]

\newenvironment{meuenumerate}
{\begin{enumerate}
  \setlength{\itemsep}{2.5pt}
  \setlength{\parskip}{-1pt}
  \setlength{\parsep}{-1pt}}
{\end{enumerate}}

\def\cyclic{\mathop{\kern0.9ex{{+}
\kern-2.2ex\raise-.28ex\hbox{\Large\hbox{$\circlearrowright$}}}}\limits}

\title{ {\huge{On the characteristic connection of \\ \textit{gwistor} space} \\ \ \ }}

\author{R. Albuquerque\footnote{Departamento de Matem\'atica da Universidade de \'Evora
and Centro de Investiga\c c\~ao em Matem\'atica e Aplica\c c\~oes (CIMA-U\'E), Rua Rom\~ao
Ramalho, 59, 671-7000 \'Evora, Portugal.}\\ \texttt{rpa@uevora.pt}}

\begin{document}

\maketitle

\begin{abstract}

We give a brief presentation of gwistor space, which is a new concept from $G_2$
geometry. Then we compute the characteristic torsion $T^\ch$ of the gwistor space of an
oriented Riemannian 4-manifold with constant sectional curvature $k$ and deduce the
condition under which $T^\ch$ is $\nac$-parallel; this allows for the classification of
the $G_2$ structure with torsion and the characteristic holonomy according to known
references. The case with the Einstein base manifold is envisaged.

\end{abstract}

\vspace*{10mm}


\vspace*{4mm}


\vspace*{10mm}

The author acknowledges the support of Funda\c{c}\~{a}o Ci\^{e}ncia e Tecnologia,
Portugal, through Centro de Investiga\c c\~ao em Matem\'atica e Aplica\c c\~oes da
Universidade de \'Evora (CIMA-U\'E) and the sabbatical grant SFRH/BSAS/895/2009.

\markright{\hfill  Albuquerque \hfill}

\vspace*{10mm}

\section{Gwistor spaces with parallel characteristic torsion}


\subsection{The purpose}

It has now become clear that every oriented Riemannian 4-manifold $M$ gives rise to a $G_2$-twistor space, as well as its celebrated twistor space. The former was discovered in \cite{AlbSal1,AlbSal2} and we shall start here by recalling how it is obtained. Often we abbreviate the name $G_2$-twistor for \textit{gwistor}, as started in \cite{Alb2}. Briefly, given $M$ as before, the $G_2$-twistor space of $M$ consists of a natural $G_2$ structure on the $S^3$-bundle over $M$ of unit tangent vectors
\begin{equation*}
 SM=\bigl\{u\in TM:\ \,\|u\|=1\bigr\}
\end{equation*}
exclusively induced by the metric $g=\langle\ ,\ \rangle$ and orientation. 

We shall describe the characteristic connection $\nac$ of $SM$ in the case where $M$ is an
Einstein manifold. This guarantees the \textit{gwistor} structure is cocalibrated, an
equivalent condition. And hence the existence of that particular connection by a Theorem
in \cite{FriIva1}. Then we restrict to constant sectional curvature; we deduce the
condition under which the characteristic torsion, i.e. the torsion of the characteristic
connection, is parallel for $\nac$. Finally we are able to deduce its classification,
according with the holonomy obtained and the cases in \cite{Fri1}. The reason why we made
such restriction is that the study of the characteristic connection in the general
Einstein base case is much more difficult and we wish to present the problem. We also
remark that $G_2$ manifolds with parallel characteristic torsion are solutions to the
equations for the common sector of type II super string theory, cf.
\cite{FriIva2,FriIva1}.

The author takes the opportunity to thank the hospitality of the mathematics department of Philipps Universit\"at Marburg, where part of the research work took place. In particular he thanks Ilka Agricola and Thomas Friedrich (Humboldt Universit\"at) for raising the questions which are partly answered here and for pointing many new directions of research.


\subsection{Elements of $G_2$-twistor or gwistor space}
\label{Tiogs}

Let $M$ be an oriented smooth Riemannian 4-manifold and $SM$ its unit tangent sphere bundle. The $G_2$-twistor structure is constructed with the following briefly recalled techniques (cf. \cite{Alb2,AlbSal1,AlbSal2}).

Let $\pi:TM\rr M$ denote the projection onto $M$, let $\na^\LC$ be the Levi-Civita connection of $M$ and let $U$ be the canonical vertical unit vector field over $TM$ pointing outwards of $SM$. More precisely, we define $U$ such that $U_u=u$, $\forall u\in TM$. The Levi-Civita connection of $M$ induces a splitting $TTM\simeq\pi^*TM\oplus\pi^*TM$. The pull-back bundle on the left hand side is the horizontal subspace $\ker\pi^*\na^\LC_.U$ isomorphic to $\pi^*TM$ through $\dx\pi$. The other $\pi^*TM$, on the right, is the vertical subspace $\ker\dx\pi$. We are henceforth referring to the classical decomposition of $TTM$, as displayed in several articles and textbooks.

Restricting $\pi$ to $SM$ we have $T\,SM=H\oplus V$ where $H$ denotes the restriction of the horizontal sub-bundle to $SM$ and $V$ is such that $V_u=u^\perp\subset\pi^*TM$, thus contained on the vertical side. Every vector field over $SM$ may be written as
\begin{equation}\label{split2}
 X=X^h+X^v=X^h+\pi^*\na^\LC_XU.
\end{equation}
The tangent sphere bundle inherits a Riemannian metric, the induced metric from the metric
on $TM$ attributed to Sasaki: $\pi^*g\oplus\pi^*g$. We simply invoke this metric with the
same letter $g$ or by the brackets $\langle\ ,\ \rangle$. Then we may say that $SM$ is the
locus set of the equation $\langle U,U\rangle=1$ and indeed \eqref{split2} is confirmed:
notice $\dx\langle U,U\rangle(X)=2\langle \pi^*\na^\LC_XU,U\rangle$. There is also a
natural map
\begin{equation}\label{aplicacaoteta}
  \theta:TTM\lrr TTM  
\end{equation}
which is a $\pi^*\na^\LC$-parallel endomorphism of $TTM$ identifying $H$ isometrically
with the vertical bundle $\pi^*TM=\ker\dx\pi$ and defined as 0 on the vertical side. It
was introduced in \cite{Alb2,AlbSal1,AlbSal2}. Then we define the horizontal vector field
$\theta^tU$.

The original discovery of the gwistor space is now explained.

Each 4-dimensional vector space $(\pi^*TM)_u,\ u\in SM$, has a natural quaternionic
structure given as follows. Every vector may certainly be written as $\lambda u+X$
with $\lambda\in\R$ and $X\perp u$. Then two such vectors multiply by
\[  (\lambda_1u+X_1)\cdot(\lambda_2u+X_2)=
(\lambda_1\lambda_2-\langle X_1,X_2\rangle)u+\lambda_1 X_2+\lambda_2X_1+X_1\times X_2  \]
where the cross-product $X_1\times X_2$ is given by 
\[ \langle X_1\times X_2,Z\rangle=\pi^*\vol_M(u,X,Y,Z),\quad \forall X,Y,Z \in u^\perp .\]
A conjugation map is obvious: $\overline{\lambda u+X}=\lambda u-X$. With this metric
compatible quaternionic structure (normed algebra with unit) and with the canonical
splitting and the map $\theta$, we may apply the Cayley-Dickson process to obtain an
octonionic structure on $TTM_{|SM}$ having the vertical $U_u=u$ as generator of the reals.
The imaginary part is the tangent bundle to $SM$, with a natural $G_2=\Aut{\Octoni}$
structure. This defines gwistor space.

The tangent bundle $T\,SM$ inherits a metric connection, via the pull-back connection and still preserving the splitting, which we denote by $\na^\star$. On tangent vertical directions, due to the geometry of the 3-sphere with the round metric, we must add a correction term to the pull-back connection. That is, for any $X,Y\in\Gamma(T\,SM)$:
\begin{eqnarray}\label{correcterm}
\na^\star_Y X^v &=& \pi^*\na^\LC_Y X^v-\langle \pi^*\na^\LC_Y X^v,U\rangle U\ =\  \pi^*\na^\LC_Y X^v+\langle X^v,Y^v\rangle U.
\end{eqnarray}
We then let $\calR(X,Y)=\pi^*R(X,Y)U=R^{\pi^*\na^\LC}(X,Y)U$, which is a $V$-valued tensor. We follow the convention $R(X,Y)= [\na^\LC_X,\na^\LC_Y] -\na^\LC_{[X,Y]}$. Notice $\calR(X,Y)=\calR(X^h,Y^h)$. Finally, the Levi-Civita connection $\nag$ of $SM$ is given by
\begin{equation}\label{lcTM}
\nag_XY=\na^\star_XY-\dfrac{1}{2}\calR(X,Y)+A(X,Y)
\end{equation}
where $A$ is the $H$-valued tensor defined by
\begin{equation}\label{AlcTM}
\langle A(X,Y),Z\rangle=\dfrac{1}{2}\bigl(\langle \calR(X,Z), Y\rangle + \langle \calR(Y,Z), X\rangle\bigr),
\end{equation}
for any vector fields $X,Y,Z$ over $SM$.

There are many global differential forms on $SM$. Specially relevant are the 1- and a 2-forms given by
\begin{equation*}
\mu(X)=\langle U,\theta X\rangle\qquad\mbox{and}\qquad\beta(X,Y)=\langle \theta X,Y\rangle-\langle\theta Y,X\rangle.
\end{equation*}
One can easily deduce $\beta=-\dx\mu$.

\subsection{Structure forms of gwistor space}

The easiest way to see other differential forms of gwistor space is by taking an
orthonormal basis on a trivialised neighbourhood as follows. First we take a direct
orthonormal basis $e_0,\ldots,e_3$ of $H$, arising from another one fixed on the
trivialising open subset of $M$, such that $e_0=u\in SM$ at each point $u$, i.e.
$e_0=\theta^tU$. Then we define
\begin{equation}\label{baseadaptada}
 e_4=\theta e_1,\qquad e_5=\theta e_2,\qquad e_6=\theta e_3
\end{equation}
which completes the desired set; we say $e_0,\ldots,e_6$ is a standard or adapted frame.
Note $\theta e_0=U$, as if $u$ has the gift of ubiquity. The dual co-frame is used to
write
\begin{equation*}\label{strforms}
\begin{split}
 \mu=e^0,\quad\quad\vol=e^{0123},\quad\quad\dx\mu=e^{41}+e^{52}+e^{63},\quad\quad
\alpha=e^{456},\\
 \alpha_1=e^{156}+e^{264}+e^{345},\qquad\quad\alpha_2=e^{126}+e^{234}+e^{315},\qquad\quad\alpha_3=e^{123}.
\end{split}
\end{equation*}
These are all global well-defined forms. They satisfy the basic structure equations, cf.
\cite{Alb2}:
\begin{equation}\label{bse1} 
\begin{split}
*\alpha=\vol=\mu\wedge\alpha_3=\pi^*\vol_M,\ \ \ \ \ \ \ \ \
*\alpha_1=-\mu\wedge\alpha_2,\ \ \ \ \ \ \ \ *\alpha_2=\mu\wedge\alpha_1,\hspace{0cm}\\
*\dx\mu=\frac{1}{2}\mu\wedge(\dx\mu)^2,\ \ \ \ \ *(\dx\mu)^2=2\mu\wedge\dx\mu,\ \ \ \ \
(\dx\mu)^3\wedge\mu=6\Vol,\hspace{0cm} \\
\alpha_1\wedge\alpha_2=3*\mu=\frac{1}{2}(\dx\mu)^3,\ \ \ \ \
\dx\mu\wedge\alpha_i=\dx\mu\wedge*\alpha_i=\alpha_0\wedge\alpha_i=0,
\end{split}
\end{equation}
$\forall i=0,1,2$, where we wrote $\alpha=\alpha_0$. We use the notation $e^{ab\cdots
jk}=e^ae^b\cdots e^je^k$ and often omit the wedge product symbol, like in $(\dx\mu)^2$.

We have given the name \textit{$G_2$-twistor} or \textit{gwistor space} to the $G_2$
structure on $SM$ defined by the stable 3-form
\begin{equation*}
 \phi=\alpha-\mu\wedge\dx\mu-\alpha_2
\end{equation*}
(as explained previously, it is induced by the Cayley-Dickson process using the vector
field $U$ and the volume forms $\vol,\alpha$). Let $*$ denote the Hodge star product. Then
\begin{equation*}
 *\phi=\vol-\dfrac{1}{2}(\dx\mu)^2-\mu\wedge\alpha_1.
\end{equation*}
We know from \cite[Proposition 2.4]{Alb2} that 
\begin{equation*}\label{dphiedestrelaphi1}
\dx\phi=\calR\alpha+\underline{r}\vol-(\dx\mu)^2-2\mu\wedge\alpha_1\qquad\ \mbox{and}\
\qquad \dx*\phi=-\rho\wedge\vol
\end{equation*}
where we have set
\begin{equation}\label{dalpha}
 \calR\alpha\ =\ \sum_{0\leq i<j\leq 3}R_{ij01}e^{ij56}+R_{ij02}e^{ij64}+R_{ij03}e^{ij45},
\end{equation}
with $R_{ijkl}=\langle R(e_i,e_j)e_k,e_l\rangle$, \,$\forall i,j,k,l\in\{0,1,2,3\}$. 

Also, $\underline{r}=r(U,U)$ is a function, with $r$ the Ricci tensor, and $\rho$ is the 1-form $(\ric U)^\flat\in\Omega^0(V^*)$, vanishing on $H$ and restricted to vertical tangent directions. One may view $\rho$ as the vertical lift of $r(\ ,U)$. We continue considering the adapted frame $e_0,\ldots,e_6$ on $SM$; then
\begin{equation}\label{roUric}
 \rho=\sum_{i,k=1}^3R_{ki0k}e^{i+3}\qquad\quad\mbox{and}
\quad\qquad\underline{r}=\sum_{j=1}^3R_{j00j}.
\end{equation}
We also remark
\begin{equation}\label{outrasderivadas}
 \dx\alpha=\calR\alpha,\qquad\quad\dx\alpha_2=2\mu\wedge\alpha_1-\underline{r}\vol.
\end{equation}

We know the gwistor space $SM$ is never a geometric $G_2$ manifold. Recall that any given $G_2$-structure $\phi$ is parallel for the Levi-Civita connection if and only if $\phi$ is a harmonic 3-form. Indeed, our $\dx\phi$ never vanishes. However, an auspicious result leads us forward. $(SM,\phi)$ is cocalibrated, ie. $\delta\phi=0$, if and only if $M$ is an Einstein manifold, cf. \cite{Alb2,AlbSal1,AlbSal2}. 

The curvature of the unit tangent sphere bundle has been studied, but the Riemannian
holonomy group remains unknown in general (cf. \cite{AbaKow,Blair} and the references
therein). From the point of view of gwistor spaces, hence just on the 4-dimensional base
space, we are interested on the holonomy of the $G_2$ characteristic connection.


\subsection{The characteristic connection}
\label{Thecharconnec}

Following the theory of metric connections on a Riemannian 7-manifold $(N,\phi)$ with $G_2$ structure, cf. \cite{Agri,FriIva2,FriIva1}, the \textit{characteristic connection} consists of a metric connection with skew-symmetric torsion for which $\phi$ is parallel. If it exists, then it is unique. Formally we may write
\begin{equation*}
 \langle\nac_XY,Z\rangle=\langle\nag_XY,Z\rangle+\frac{1}{2}T^\ch(X,Y,Z)
\end{equation*}
where $g$ denotes the metric and $\nag$ the Levi-Civita connection. If $\phi$ is cocalibrated, then such $T^\ch$ exists; it is given by
\begin{equation}\label{torcaocaracterisgeral}
 T^\ch=*\dx\phi-\frac{1}{6}\langle\dx\phi,*\phi\rangle\phi
\end{equation}
cf. \cite[Theorems 4.7 and 4.8]{FriIva1}\footnote{Notice we use a different orientation than that in \cite{FriIva1}. Therefore, we have to replace $\ast$ by $-\ast$ in formulas given there.}.

We recall there are three particular $G_2$-modules decomposing the space $\Lambda^3$ of
3-forms (cf. \cite{Besse,Bryant2,FerGray}). They are $\Lambda^3_1,\ \Lambda^3_7,\
\Lambda^3_{27}$, with the lower indices standing for the respective dimensions. In the
same reasoning, $\Lambda^2=\Lambda^2_7\oplus\Lambda^2_{14}$. Thus, by Hodge duality,
$\dx\phi$ has three invariant structure components and $\delta\phi$ has two. In gwistor
space we have proved the latter vanish altogether, or not, with $\rho$, given in
\eqref{roUric}. The analysis of the tensor $\dx\phi$ is struck with the never-vanishing
component in $\Lambda^3_{27}$. It is of pure type $\Lambda^3_{27}$ if and only if $M$ is
an Einstein manifold with Einstein constant $-6$ (see \cite[Theorem 3.3]{Alb2}). 

Apart from a Ricci tensor dependent component, the curvature tensor of $M$ contained in $\dx\phi=\calR\alpha+\cdots$ remains much hidden in the $\Lambda^3_{27}$ subspace.

We have deduced a formula for the Levi-Civita connection $\nag$ of $SM$, shown in
\eqref{lcTM}. The characteristic connection $\nac$ is to be deduced here in the
cocalibrated case given by a constant sectional curvature metric on $M$. In our opinion,
this analysis corroborates the correct choice of techniques in dealing with the equations
of gwistor space.

\subsection{Characteristic torsion of gwistor space}

Let us start by assuming $M,g$ is an Einstein manifold with Einstein constant $\lambda$. Such condition is given by any of the following, where $\lambda$ is a priori a scalar function on $M$:
\begin{equation*}
 r=\lambda g\quad \Leftrightarrow\quad \ric U=\lambda U\quad \Leftrightarrow\quad\underline{r}=\lambda.
\end{equation*}
In our setting it is also equivalent to $\dx*\phi=0$. Then $\lambda$ is a constant.
\begin{prop}
The characteristic connection $\nac=\nag+\frac{1}{2}T^\ch$ of $SM$ is given by
\begin{equation*}
 T^\ch=*(\calR\alpha)+\frac{2\lambda-6}{3}\alpha+\frac{\lambda}{3}\mu\wedge\dx\mu+
\frac{\lambda}{3}\alpha_2.
\end{equation*}
Moreover, $\delta T^\ch=0$.
\end{prop}
\begin{proof}
We have by \eqref{dalpha} and some computations
\begin{equation*}
 \langle\calR\alpha,
*\phi\rangle\Vol=\calR\alpha\wedge\phi=-\calR\alpha\wedge(\mu\dx\mu+\alpha_2)=\lambda\Vol.
\end{equation*}
Also $\underline{r}\vol\phi=\lambda\Vol,\ \,-(\dx\mu)^2\phi=\mu\wedge(\dx\mu)^3=6\Vol,\
-2\mu\wedge\alpha_1\wedge\phi=2\mu\wedge\alpha_1\wedge\alpha_2=6\Vol$. Hence
$\langle\dx\phi,*\phi\rangle=2(\lambda+6)$. One finds helpful identities in \eqref{bse1}.
Since $*\dx\phi=*\calR\alpha+\lambda\alpha-2\mu\wedge\dx\mu-2\alpha_2$, we get from
\eqref{torcaocaracterisgeral}
\begin{eqnarray*}
 T^\ch&=&*\dx\phi-\frac{2}{6}(\lambda+6)\phi\\
&=&*\calR\alpha+\lambda\alpha-2\mu\wedge\dx\mu-2\alpha_2-(\frac{\lambda}{3}
+2)(\alpha-\mu\wedge\dx\mu-\alpha_2)
\end{eqnarray*}
and the first part of the result follows. From the first line we immediately see $\dx*T^\ch=0$.
\end{proof}

Until the rest of this section we assume $M$ has constant sectional curvature $k$, so that $R_{ijkl}=k(\delta_{il}\delta_{jk}-\delta_{ik}\delta_{jl})$
with $k\in\R$ a constant. Then by \eqref{dalpha}
\begin{equation*}
  \calR\alpha=-k\mu\wedge\alpha_1.
\end{equation*}
In particular,
\begin{equation*}
\dx\phi=3k\vol-(\dx\mu)^2-(k+2)\mu\wedge\alpha_1.
\end{equation*}
Henceforth $\lambda=\underline{r}=3k$ and $*\calR\alpha=-k\alpha_2$, and the following result is immediate.
\begin{prop}
The characteristic torsion of the characteristic connection is given by
\begin{equation}\label{torsaocaracteristicadeespacoforma}
 T^\ch=2(k-1)\alpha+k\mu\wedge\dx\mu .
\end{equation}
\end{prop}
Taking formulas \eqref{lcTM} and \eqref{AlcTM}, the next Propositions are the result of
simple computations.
\begin{prop}
For any $X,Y\in TSM$:
\begin{meuenumerate}
\item $\calR(X,Y)=k(\langle\theta Y,U\rangle\theta X-\langle\theta X,U\rangle\theta Y)$; or simply $\calR=k\theta\wedge\mu$
\item $A(X,Y)=\frac{k}{2}\bigl(\langle\theta X,Y\rangle\theta^tU+\langle\theta Y,X\rangle\theta^tU-\mu(X)\theta^tY-\mu(Y)\theta^tX\bigr)$.
 \end{meuenumerate}
\end{prop}
We also omit the proof of the next formulas. These are the application of the general case treated in \cite[Proposition 2.2]{Alb2} to our situation with $\calR$ and $A$ given just previously.
\begin{prop}\label{dericovari}
For any $X\in TSM$ we have:
\begin{meuenumerate}
\item $\nag_X\theta^tU=\frac{2-k}{2}\theta^tX-\frac{k}{2}(\theta X-\mu(X)U)$
\item $\nag_X\vol=A_X\cdot\vol=\frac{k}{2}\bigl(\mu(X)\mu\wedge\alpha_2
 -(\theta X)^\flat\wedge\alpha_3-(X^\flat\circ\theta)\wedge \alpha_3\bigr)$
\item $\nag_X\alpha=\frac{k}{2}\bigl(\mu\wedge(\theta X)\lrcorner\alpha-\mu(X)\alpha_1\bigr)$
\item $\nag_X\mu =\frac{2-k}{2}X^\flat\circ\theta-\frac{k}{2}(\theta X)^\flat$
\item $\nag_X\dx\mu=\frac{k}{2}\mu\wedge\bigl({(X^h)}^\flat-{(X^v)}^\flat\bigr)$
\item $\nag_X\alpha_1=k\mu(X)\bigl(\frac{3}{2}\alpha-\alpha_2\bigr)+\mu\wedge\bigl(\frac{k-2}{2}X\lrcorner\alpha+\frac{k}{2}(\theta X)\lrcorner\alpha_1\bigr)$
\item $\nag_X\alpha_2=k\mu(X)\bigl(\alpha_1-\frac{3}{2}\alpha_3\bigr)+\mu\wedge\bigl(\frac{k-2}{2}X^v\lrcorner\alpha_1+\frac{k}{2}X\lrcorner\alpha_3\bigr)$
\item $\nag_X\alpha_3=\frac{2-k}{2}(\theta^tX)\lrcorner\vol+(\theta^tU)\lrcorner A_X\cdot\vol=\frac{2-k}{2}(\theta^tX)\lrcorner\vol+\frac{k}{2}\mu(X)\alpha_2$.
\end{meuenumerate}
\end{prop}
We may now deduce:
\begin{equation*}
\begin{split}
 \nag_X\phi\ =\ \nag_X\alpha-\nag_X\mu\wedge\dx\mu-\mu\wedge\nag_X\dx\mu
-\nag_X\alpha_2\qquad\qquad\qquad\\
=\ \frac{k}{2}\mu\wedge(\theta X)\lrcorner\alpha-\frac{3k}{2}\mu(X)\alpha_1+
\bigl(\frac{k-2}{2}X^\flat\circ\theta+\frac{k}{2}(\theta X)^\flat\bigr)\wedge\dx\mu\\
+\frac{3k}{2}\mu(X)\alpha_3-\frac{k}{2}\mu\wedge X\lrcorner\alpha_3-\frac{k-2}{2}\mu\wedge X^v\lrcorner\alpha_1.\qquad
\end{split}
\end{equation*}
A computation confirms that $\nac\phi=0$ with $\nac=\nag+\frac{1}{2}T^\ch$ and $T^\ch$
given by \eqref{torsaocaracteristicadeespacoforma}.

We are now in position to compute $\nac T^\ch$.
\begin{teo}\label{Torsaocaracteristicaparalelaemespacosforma}
 Let $M$ be an oriented Riemannian 4-manifold of constant sectional curvature $k$. The characteristic connection $\nac$ of the associated gwistor space satisfies
\begin{eqnarray*}
 \nac_{X}T^\ch &=& k(k-1)X^v\lrcorner(\mu\wedge\alpha_1-\frac{1}{2}(\dx\mu)^2).
\end{eqnarray*}
In particular, $SM$ has parallel torsion if and only if $k=0$ or $k=1$.
\end{teo}
\begin{proof}
For any direction $X\in TSM$ and using the cyclic sum in three vectors,
\begin{equation*}
\nac_XT^\ch = \nag_XT^\ch-\cyclic T^\ch(\dfrac{1}{2}T^\ch_X\ ,\ ,\ )=  \nag_XT^\ch-\frac{1}{2}\sum_{j=0}^6 T^\ch(X,\ ,e_j)\wedge T^\ch(e_j,\ ,\ ).
\end{equation*}
Since $T^\ch=2(k-1)\alpha+k\mu\wedge\dx\mu$, we get
\begin{eqnarray*}
\nac_XT^\ch &=& \nag_XT^\ch-\frac{k^2}{2}\dx\mu(X,\ )\wedge\dx\mu+
k(k-1)X\lrcorner(\mu\wedge\alpha_1).
\end{eqnarray*}
Now we have from Proposition \ref{dericovari}
\[ \nag_XT^\ch=(k-1)k\bigl(\mu\wedge(\theta X)\lrcorner\alpha 
-\mu(X)\alpha_1\bigr)-\bigl(k\frac{k-2}{2}X^\flat\circ\theta+ \frac{k^2}{2}(\theta
X)^\flat\bigr)\wedge\dx\mu \]
and then we see easily that $\nac_{X}T^\ch=0$ for $X\in H$. Taking a vertical direction $X$, the desired formula for the covariant derivative of $T^\ch$ is achieved.
\end{proof}

Let us now see the decompositions under $G_2$ representations referred in section \ref{Thecharconnec}, in the case under appreciation. Recall $\dx\phi$ has no $\Lambda^3_7$ component. Since $\delta\phi=0$, there are no $\Lambda^2_7,\Lambda^2_{14}$ components either. By results in \cite[Proposition 3.6]{Alb2}, the $\Lambda^3_1,\Lambda^3_{27}$ parts are
\[ \dx\phi=\frac{6}{7}(k+2)*\phi+ 
*\frac{1}{7}\bigl((15k-12)\alpha+(6k-2)\mu\wedge\dx\mu-(k+2)\alpha_2\bigr) . \]
Regarding $T^\ch$ it is coclosed. Now a characteristic connection with closed torsion is
called a \textit{strong $G_2$ with torsion}, denoted SG${}_2$T in \cite{ChiSwa}. We may
retain the following result.
\begin{prop}
 The gwistor space arising from constant curvature on $M$ has an SG${}_2$T connection if
and only if $k=0$.
\end{prop}
We have the decompositions $T^\ch=-\frac{k+2}{7}\phi+\tau_3^{T^\ch}$ and 
\[ \dx T^\ch=k(\dx\mu)^2-2k(k-1)\mu\wedge\alpha_1=\frac{6}{7}k(k-2)*\phi+*\tau_3^{\dx
T^\ch},\]
where $\tau_3^{T^\ch},\tau_3^{\dx T^\ch}$, sitting in
$\Lambda^3_{27}=\ker(\cdot\wedge\phi)\,\cap\,\ker(\cdot\wedge*\phi)$, are given by
\begin{equation*}
 \begin{split}
  \tau_3^{T^\ch} = \frac{1}{7}
\bigl((15k-12)\alpha+(6k-2)\mu\wedge\dx\mu-(k+2)\alpha_2\bigr),\\
\tau_3^{\dx T^\ch} =  \frac{2k}{7}
\bigl((6-3k)\alpha+(1+3k)\mu\wedge\dx\mu+(1-4k)\alpha_2\bigr).
 \end{split}
\end{equation*}

\section{Holonomy of the `parallel torsion'}

\subsection{Results on the Stiefel manifold $V_{l,2}$}
\label{TSmV_l2}

Theorem \ref{Torsaocaracteristicaparalelaemespacosforma} leads to the consideration of two
distinct cases. We start with $k=1$.

Since our results so far are local, we assume $M$ is simply-connected and complete. As it is well known, $SM$ with $M=S^4_1$, the radius 1 sphere, agrees with the Stiefel manifold $SO(5)/SO(3)=V_{5,2}$. Recall that transitivity of the action by isometries induced on the tangent sphere bundle of a Riemannian symmetric space is exclusive to all rank 1 spaces, cf. \cite[Proposition 10.80]{Besse}. In particular, in dimension 4, we are left with $S^4$, $\Proj^2(\C)$, the real hyperbolic space $H^4$ and the hyperbolic Hermitian space $\C H^2$.

We thus study briefly the space $V_{l,2}$, the unit tangent sphere bundle of $S^{l-1}$ with $l>2$. In the sequel, we let the name Stiefel manifold refer just to $V_{l,2}$ (with the index 2 fixed). Firstly, the Stiefel manifolds are simply-connected for $l\geq5$. The following results are due to Stiefel and to Borel, cf. \cite[Proposition 10.1]{Bor}:
\begin{equation*}
\begin{cases}
H^*(V_{l,2},\Z)=H^*(S^{l-1}\times S^{l-2},\Z)\qquad\mbox{if $l$ is even}\\
H^0(V_{l,2},\Z)=H^{2l-3}(V_{l,2},\Z)=\Z,\quad H^{l-1}(V_{l,2},\Z)=\Z_2\qquad\mbox{if $l$ is odd}\\
H^*(V_{l,2},\Z_2)=H^*(S^{l-1}\times S^{l-2},\Z_2)=\wedge\{x_{l-1},x_{l-2}\}.
\end{cases}
\end{equation*}
$\wedge$ stands for the free multiplicative exterior algebra generated on the given $x_j$ of degree $j$. We also have the additive isomorphism $H^*(V_{l,2},\Z_2)=H^*(S^{l-1},\Z_2)\otimes H^*(S^{l-2},\Z_2)$. Moreover, $V_{l,2}$ is a rational homology sphere for $l$ odd, cf. \cite{BoyGalNaka1}.
Now, we may deduce that 
\[  w(SS^{l-1})=\sum\pi^*w_i^2  \]
where $S^{l-1}$ is the base manifold and $\pi$ is the projection. There is a general formula in \cite{Alb4}. It is well known that $w(S^k)=\sum_{i\geq0} w_i=1$ for all $k$. Hence the following result for which we do not know a reference.
\begin{prop}\label{classesSWdeVl2}
The total Stiefel-Whitney class of $V_{l,2}$ is 1. In particular, this space is orientable and admits a spin structure.
\end{prop}
Now, regarding the Riemannian structure from a slightly general picture, let us see how we are driven to $V_{l,2}=SS^{l-1}$ with the metric induced from the Sasaki metric of the tangent bundle, cf. section \ref{Tiogs}.

First we recall from \cite{Agri,BoyGalNaka1,BoyGalMatzeu,FriIva1} what is the natural geometric notion concerned with a Riemannian reduction from the Lie group $SO(2n+1)$ to the structure group $U(n)$. A metric almost contact manifold consists of a Riemannian manifold $({\cal S},\tilde{g})$ together with a 1-form $\eta$, a vector field $\xi$ and an endomorphism $\varphi\in\Gamma(\End{T\cal S})$ satisfying the relations: $\forall X,Y\in T\cal S$
\begin{equation*}
\begin{split}
\eta(\xi)=1,\qquad\varphi^2=-1+\eta\otimes\xi,\hspace{2.2cm}\\
\tilde{g}(\varphi X,\varphi Y)=\tilde{g}(X,Y)-\eta(X)\eta(Y),\qquad\varphi(\xi)=0.
\end{split}
\end{equation*}
If furthermore $\dx\eta=2F$, where $F(X,Y)=\tilde{g}(X,\varphi Y)$, then we have a metric contact structure. If the CR-structure defined by the distribution $\cal D=\ker\eta$ is integrable, then we have a so called normal contact structure. The integrability condition is the vanishing of a certain Nijenhuis tensor of the almost complex structure $J=\varphi_{|\cal D}$. If $\xi$ is a Killing vector field, i.e. ${\cal L}_\xi\tilde{g}=0$, then we say we have a K-contact structure. Since on a contact structure we have ${\cal L}_\xi F=0$, the K-contact equation is assured equivalently by ${\cal L}_\xi\varphi=0$. A normal K-contact structure is known as a Sasakian structure; then $\cal S$ is called a Sasakian manifold.

The K-contact condition is equivalent to $\nag_X\xi=-\varphi(X),\ \forall X\in T\cal S$. A K-contact structure is normal (and thence the manifold is Sasakian) if furthermore (cf. \cite{FriKath})
\begin{equation}\label{Sasakianequation}
(\nag_X\varphi)(Y)=\tilde{g}(X,Y)\xi-\eta(Y)X.
\end{equation}

Now let $M$ be a Riemannian manifold of dimension $m=n+1$.  Y. Tashiro has shown the unit
tangent sphere bundle $SM$ (of dimension $2n+1$) has a metric contact structure. It is
given, in present notation, by $\tilde{g}=\frac{1}{4}g$, $\eta=\frac{1}{2}\mu,\
\xi=2\theta^tU$ and $\varphi=\theta-U\mu-\theta^t$. Notice $g$ is the Sasaki metric and
$\theta$ is the map in \eqref{aplicacaoteta}. We have, by \eqref{outrasderivadas} easily
generalized to any dimension,
\[ F(X,Y):=\frac{1}{4}g(X,\varphi Y)=\frac{1}{4}(\langle X,\theta Y\rangle-\langle\theta
X,Y\rangle)=\frac{1}{4}\dx\mu(X,Y),\]
so $\dx\eta=2F$ as expected. Tashiro also proved the following \cite[Theorem 9.3]{Blair}:
the contact metric structure on $SM$ is a K-contact structure if and only if $(M,g)$ has
constant sectional curvature 1. And then deduces $SM$ is Sasakian. The proof goes as
follows: notice $\nag=\na^{\tilde{g}}$ is given in \eqref{lcTM}. Then we find
\[  \langle\nag_X\xi,Y^v\rangle=-\frac{1}{2}\langle\calR_{X,\xi},Y^v\rangle=
-\langle R_{X^h,\theta^tU}U,Y^v\rangle.  \]
So, just looking at the vertical part of the equation $\nag_X\xi=-\varphi(X)$, on the base
manifold it reads $\langle R(X,u)u,Y\rangle=\langle X,Y\rangle,\ \forall X,Y\in TM\cap
u^\perp$. Clearly, this means constant sectional curvature 1. The horizontal part of the
equation gives the same result. The reciprocal is also easy, and the Sasakian condition
follows. Moreover, in this case the Sasakian equation \eqref{Sasakianequation}) alone
implies the round curvature 1.

A contact manifold $({\cal S},\tilde{g},\eta,\xi,\varphi)$ is said to be $\eta$-Einstein if its Ricci tensor can be written as $\ric_{\tilde{g}}(X,Y)=\lambda\tilde{g}+\nu\eta\otimes\eta$ with $\lambda,\nu$ constants (cf. \cite{BoyGalMatzeu,Oku}).

We compute, with methods as found in \cite{AbaKow}, that the contact manifold $(SM,\tilde{g},\eta,\xi,\varphi)$ verifies equalities
\begin{equation}\label{ricci_g}
\begin{split}
\ric_{\tilde{g}}(X,Y)=\ric_g(X,Y)=  \hspace{4.5cm}\\
=((m-1)k-\frac{k^2}{2})\langle X^h,Y^h\rangle+(m-2+\frac{k^2}{2})\langle X^v,Y^v\rangle+\frac{k^2}{2}(2-m)\mu(X)\mu(Y)
\end{split}
\end{equation}
if $M$ has constant sectional curvature $k$.
\begin{prop}[\cite{ChaiChunParkSek}]
Assuming constant sectional curvature $k$, the contact manifold $SM$ is $\eta$-Einstein if and only if $k=1$ or $k=m-2$.
\end{prop}
This result was also deduced by \cite{ChaiChunParkSek}. In the Sasakian case $k=1$ notice the formula $\lambda+\nu=2n$, as theoretically expected (\cite[Lemma 7]{BoyGalMatzeu}).

In \cite{FriIva1,Oku} we have the notion of contact connection on a contact manifold, i.e. a linear connection on $\cal S$ such that
\begin{equation*}
\na\tilde{g}=0,\qquad\na\eta=0,\qquad\na\varphi=0.
\end{equation*}
\cite[Theorem 8.4, case 1]{FriIva1} guarantees that any Sasakian manifold admits a contact connection with totally skew-symmetric torsion given by
\begin{equation}\label{TorsaoContactoSasakian}
 T=\eta\wedge\dx\eta . 
\end{equation}
Moreover, $T$ is parallel for such $\na=\nag+\frac{1}{2}T$, which is unique ---  so it is called \textit{the} characteristic connection of the normal contact structure. In general, cf. \cite[Theorem 8.2]{FriIva1}, this contact connection with skew-symmetric torsion exists if and only if the Nijenhuis tensor is skew-symmetric and $\xi$ is a Killing vector field.

In sum, Tashiro's results on $SM$ led us to the case of integrable geometries, the
homogeneous Sasakian space $V_{l,2}$, where $l=m+1=n+2$, with metric $\frac{1}{4}g$ and
Ricci curvature tensor $\ric_g=(m-\frac{3}{2})g+\frac{2-m}{2}\mu\otimes\mu$. This space
admits a characteristic contact connection ($T$ is the same viewed as a $(2,1)$-tensor)
$\na=\nag+\frac{1}{2}\mu\wedge\dx\mu$.
And there is no simply-connected Riemannian manifold besides $S^4$ whose unit tangent sphere bundle admits a characteristic contact connection. For the existence assures the manifold is K-contact.

To complete the picture, the characteristic foliation ${\cal F}_\xi$ determined by $\xi$, hence with 1 dimensional leaves, gives a projection onto the Grassmannian of oriented 2-planes in $\R^l$, a complex quadric, $V_{l,2}\rr{\tilde{\mathrm{Gr}}}_{l,2}$.

Starting from a K\"ahler-Einstein manifold $(X^{2n},\overline{g},\overline{J})$ of scalar curvature $4n(n+1)$, it is  shown in \cite[pag. 83]{BFKG} how to construct Einstein-Sasakian metrics on an associated $S^1$-bundle $\pi:{\cal S}\rr X^{2n}$: the bundle whose first Chern class is $c_1=\frac{1}{A}c_1(X^{2n})$ where $A$ is the maximal integer such that $\frac{1}{A}c_1(X^{2n})$ is an integral cohomology class. Moreover, $\cal S$ is simply connected and admits a spin structure (cf. Proposition \ref{classesSWdeVl2}). The 1-form $\eta$ is induced by the associated $U(1)$-connection, so that $\dx\eta$ is essentially the K\"ahler form of $X^{2n}$.

The example of the Stiefel manifold is already mentioned in \cite{BFKG}, as noticed by \cite{BoyGalNaka1}.

\subsection{Holonomy of the characteristic contact connection of $V_{l,2}$}

We may now continue our study of the gwistor space of the 4-sphere with the canonical Sasaki metric.
\begin{prop}\label{characteristiccoonection}
The characteristic connection $\nac$ of the $G_2$-twistor space $(V_{5,2},g,\phi)$ is
given by the torsion $T^\ch=\mu\wedge\dx\mu$ and its holonomy is contained in $SU(3)$.
Thus coincides with the contact metric connection. The torsion is parallel.
\end{prop}
\begin{proof}
By the results given in \eqref{TorsaoContactoSasakian} we find a contact connection with
skew-symmetric torsion $T=\mu\wedge\dx\mu$ (contracting now with the metric $g$).
Additionally we have that $T$ is $\na$-parallel. We remark that $\mu\wedge\na\dx\mu=0$.
Computing $\na\alpha=(\nag+\frac{1}{2}T)\alpha$, applying Proposition \ref{dericovari} and
the usual technique, we find $\alpha$ is parallel. Since $\na\varphi=0$ and
\[ \alpha_2=\frac{1}{2}\alpha\circ(\theta\wedge\theta\wedge1)=
  \frac{1}{2}\alpha\circ(\varphi\wedge\varphi\wedge1) , \]
(cf. \cite{Alb2} for this notation and remarks on differentiation) we get
\[ \na\alpha_2=0 . \]
Hence $\na\phi=\na(\alpha-\mu\wedge\dx\mu-\alpha_2)=0$ and therefore the (unique)
$SU(3)\subset G_2$ connection $\na$ with totally skew-symmetric torsion is the
characteristic connection of the gwistor structure, $\na=\nac$.
\end{proof}
Of course $T^\ch$ above agrees with the result found in
\eqref{torsaocaracteristicadeespacoforma} for sectional curvature 1.

Now, the formulas from \cite{FriIva1} for the curvature of the characteristic connection
are combined with the Riemannian curvature. So it is important to recall the references on
the latter. There is a long literature on results about the sectional, Ricci and scalar
curvatures of the Sasaki metric on the tangent sphere bundle of any given Riemannian
manifold. The techniques are those from e.g. \cite{AbaKow} and several other references
therein, where Einstein metrics are found (interesting enough, for $V_{l,2}$ we also have
$SO(l)$-invariant Einstein metrics given in \cite{Arva}, which use the method below and
recur to results of Wang). 

We follow homogeneous space theory to determine the holonomy of the characteristic
connection.

Let $n,m$ be integers such that $l=m+1=n+2$ (as in section \ref{TSmV_l2}), let $K=SO(l),\ H=SO(n),\ \g=\sol(l), \h=\sol(n)$. Now we consider the trivial embedding $H\subset K$. So we may decompose $\g=\h\oplus\m$ with $\m$ the subspace of matrices having 0 where $\h$ falls. Since $[\h,\m]\subset\m$ and $H$ is connected, we have a reductive homogeneous space $V_{l,2}=K/H$. Then the tangent vector bundle of $K/H$ arises from the canonical principal $H$-bundle, associated to $\m$. Let $D_{ij}$ be the matrix with 0 everywhere except in position $(i,j)$ where it has a 1. We have a canonical basis of $\g$ given by
\[ E_{ij}=D_{ij}-D_{ji},\qquad1\leq i<j\leq l .   \]
The vectors $e_{0}=E_{m,l}$ and $e_i=E_{i,l},\ e_{i+n}=E_{i,m},\ \,1\leq i\leq n$
constitute a basis of $\m$, which we may take to be an orthonormal basis of a
$K$-invariant Riemannian metric, cf. \cite{Kath1,WangZil}. Compare also with formula
\eqref{baseadaptada}, i.e. the adapted frame of $G_2$-twistor space.

We recall the canonical connection $\na$ of $K/H$ is given by $\na_{e_a}e_b=0,\ \forall a,b$ such that $0\leq a,b\leq 2n+1$. Its torsion satisfies $T^\na(X,Y)=-[X,Y]_\m$, where the index denotes the component in $\m$, cf. \cite{KobNomi}.

The new metric corresponds with the Sasaki metric of $SS^m$ introduced in section \ref{Tiogs} and generalised to any dimension. Indeed, the embedding $SO(n)\subset SO(m)\subset SO(l)$ induces the respective decomposition of $\h$, to which the Levi-Civita connection of the sphere corresponds. The horizontal and vertical subspace decomposition is clear.
\begin{teo}
The characteristic contact connection $\nac=\nag+\frac{1}{2}\mu\wedge\dx\mu$ on $V_{l,2}$
coincides with the invariant canonical connection. Moreover, $\nac$ is complete and its
holonomy group is $SO(n)$.
\end{teo}
\begin{proof}
Here we refer just to Chapter X of \cite[Volume II]{KobNomi}. First recall from
\cite[Proposition 2.7]{KobNomi} that every $K$-invariant tensor is parallel for the
(invariant) canonical connection. By the way they were defined, the tensors
$g,\alpha,\theta,\varphi,\mu,\dx\mu,\xi$ are all clearly $K$-invariant. Also the torsion
$T^\na(X,Y,Z)=-g([X,Y],Z)=g(Y,[X,Z])$ is totally skew-symmetric. Hence the result follows
by uniqueness of the characteristic connection. The theory says the canonical connection
$\na$ is complete and what its holonomy Lie subalgebra is.
\end{proof}
Interesting enough, one may check the identity on a triple of vectors on $\m$
\begin{equation*}
 \mu\wedge\dx\mu(X,Y,Z)=-\langle[X,Y],Z\rangle.
\end{equation*}

Finally, we return to gwistor space.
\begin{coro}
The holonomy of the characteristic $G_2$-connection on the unit sphere bundle of the
4-sphere $(V_{5,2})$ is complete and equal to $SO(3)$.
\end{coro}
This allows us to look for the classification of the $G_2$-twistor
space $V_{5,2}$ with characteristic holonomy algebra $\mathfrak{hol}(\nac)\subset\g_2$,
corresponding to a $G_2$-connection with parallel skew-symmetric torsion, as described in
\cite{Fri1}. We arrive precisely to the case of \cite[Theorem 7.1]{Fri1} (with a certain
$c$ in that reference equal to $1/7$), which comes form the Lie subalgebra
$\sol(3)\subset\sul(3)\subset\g_2$.

The characteristic curvature tensor is given by $R^\ch(X,Y)Z=-[[X,Y]_\h,Z]$ or by
\begin{equation*}
R^\ch=-\frac{1}{2}(2S_1\otimes S_1+S_2\otimes S_2+S_3\otimes S_3)
\end{equation*}
as results from \cite{Fri1} (or \cite{KobNomi}), with the $S_i$ being generators of
$\sol(3)\subset\g_2$. Formulas for $\ric_g$ found in \eqref{ricci_g} match precisely with
those given in the new reference.

\subsection{Holonomy of the characteristic connection of $S\R^4$}

The characteristic connection on the gwistor space of a flat 4-dimensional space
also has parallel torsion; harmonic as proved earlier. The $G_2$-twistor
structure verifies $\dx\phi=-(\dx\mu)^2-2\mu\wedge\alpha_1$.
The simply-connected germ of such gwistor, say $SM=\R^4\times S^3$, was described in
\cite{Alb2} with canonical coordinates, inducing a trivial framing. Recall from
\eqref{torsaocaracteristicadeespacoforma} that the torsion of the characteristic $G_2$
metric connection is $T^\ch=-2\alpha$.
\begin{prop}
Let $M$ be an oriented flat Riemannian 4-manifold. Then the characteristic $G_2$
connection $\nac=\nag-\alpha$ on the gwistor space $SM$ is flat.
\end{prop}
\begin{proof}
The Levi-Civita connection of $TM$ is the flat connection $\pi^*\na^\LC=\dx$ duplicated
for $TTM=\pi^*TM\oplus\pi^*TM$. Then the Levi-Civita connection of $SM$ is just the
connection $\na^\star=\nag$ written in \eqref{correcterm}. By Gauss formula,
$R^g(X,Y,Z,W)=\langle X^v,W^v\rangle\langle Y^v,Z^v\rangle-
\langle X^v,Z^v\rangle\langle Y^v,W^v\rangle$. Using an adapted frame, cf.
\eqref{baseadaptada},
\[ R^g=-(e^{45}\otimes e^{45}+e^{56}\otimes e^{56}+e^{64}\otimes e^{64}).\] 
On the other hand, a formula in \cite{FriIva1} says
\[ R^\ch=R^g+\frac{1}{4}\sum(e_i\lrcorner T^\ch)\otimes(e_i\lrcorner
T^\ch)+\frac{1}{4}\sum (e_i\lrcorner T^\ch)\wedge(e_i\lrcorner T^\ch). \]
Thence, since $T^\ch=-2e^{456}$, we have 
\[  \frac{1}{4}\sum(e_i\lrcorner T^\ch)\otimes(e_i\lrcorner T^\ch)=e^{45}\otimes
e^{45}+e^{56}\otimes e^{56}+e^{64}\otimes e^{64}=-R^g. \]
and clearly $\sum (e_i\lrcorner T^\ch)\wedge(e_i\lrcorner T^\ch)=0$.
\end{proof}
Notice both metric connections preserve the Riemannian splitting. On the vertical side,
the connection $\nac$ is the invariant $SO(3)$-connection with skew-symmetric torsion
$-2\alpha$ described e.g. in \cite[Remark 2.1]{Agri} and known to \'E. Cartan.


\vspace{2cm}


\medskip

\begin{thebibliography}{10}



\bibitem{AbaKow}
M. T. K. Abbassi and O. Kowalski,
{\em On Einstein Riemannian $g$-natural metrics on unit tangent sphere bundles},
Ann. Global Anal. Geom., 38 (1) (2010), 11--20.

\bibitem{Agri}
I. Agricola,
{\em The Srn\'\i\ lectures on non-integrable geometries with torsion},
Archi. Math. (Brno), Tomus 42 (2006), Suppl., 5--84.


\bibitem{Alb2}
R. Albuquerque,
{\em On the $G_2$ bundle of a Riemannian 4-manifold},
J. Geom. Phys. 60 (2010), 924--939.

\bibitem{Alb4}
R. Albuquerque,
{\em Homotheties and topology of tangent sphere bundles},
to appear.


\bibitem{AlbSal1}
R. Albuquerque and I. Salavessa,
{\em The $G_2$ sphere of a 4-manifold},
Monatsh. Math. 158, Issue 4 (2009), 335--348. 


\bibitem{AlbSal2}
R. Albuquerque and I. Salavessa,
{\em Erratum to: The $G_2$ sphere of a 4-manifold},
Monatsh. Math. 160, Issue 1 (2010), 109--110.


\bibitem{Arva}
A. Arvanitoyeorgos,
{\em $SO(n)$-invariant Einstein metrics on Stiefel manifolds},
Differential Geom. and Appl., Proc. Conf., Aug. 28 - Sept. 1, 1995, Brno Czech Republic,
Masaryk University, Brno, 1--5.


\bibitem{BFKG}
H. Baum, Th. Friedrich, R. Grunewald and I. Kath, 
{\em Twistors and Killing Spinors on Riemannian Manifolds},
Teubner-Texte f\"ur Mathematik, vol. 124, Teubner, Stuttgart, Leipzig, 1991.


\bibitem{Besse}
A. L. Besse,
{\em Einstein Manifolds},
Springer-Verlag Berlin Heidelberg, 1987.


\bibitem{Blair}
D. Blair,
{\em Riemannian geometry of contact and symplectic manifolds},
Progress in Mathematics, vol. 203, Birkha\"user Boston Inc., Boston, MA, 2002.


\bibitem{Bor}
A. Borel,
{\em Sur La Cohomologie des Espaces Fibres Principaux et des Espaces Homogenes de Groupes
de Lie Compacts},
Annals of Math., Second Series, Vol. 57, No. 1 (1953), 115--207.


\bibitem{BoyGalNaka1}
Ch. P. Boyer, K. Galicki and M. Nakamaye,
{\em Einstein metrics on rational homology 7-spheres},
Ann. Inst. Fourier, Grenoble, 52, 5 (2002), 1569--1584.


\bibitem{BoyGalMatzeu}
Ch. P. Boyer, K. Galicki and P. Matzeu,
{\em On Eta-Einstein Sasakian geometry},
Comm. Math. Phys. 1, 262 (2006), 177--208.


\bibitem{Bryant2}
R. L. Bryant,
{\em Some remarks on $G_2$ structures},
Proceedings of the 2004 Gokova Conference on Geometry and Topology (May, 2003).


\bibitem{ChiSwa}
S. G. Chiossi and A. Swann,
{\em $G_2$-structures with torsion from half-integrable nilmanifolds},
J. Geom. Phys. 54, Issue 3 (2010), 262--285.


\bibitem{ChaiChunParkSek}
Y. D. Chai, S. H. Chun, J. H. Park and K. Sekigawa,
{\em Remarks on $\eta$-Einstein unit tangent bundles},
Monatsh. Math. 155 (2008), 31--42.

\bibitem{FerGray}
M. Fern\'andez and A. Gray, 
{\em Riemannian manifolds with structure group $G_2$}, 
Ann. Mat. Pura Appl. 4 132 (1982), 19--45. 

\bibitem{Fri1}
Th. Friedrich,
{\em $G_2$-manifolds with parallel characteristic torsion},
Differential Geom. Appl. 25 (2007), 632--648.

\bibitem{FriIva2}
Th. Friedrich and S. Ivanov,
{\em Killing spinor equations in dimension 7 and geometry of integrable $G_2$-manifolds},
J. Geom. Phys. 48 (2003), 1--11.

\bibitem{FriIva1}
Th. Friedrich and S. Ivanov,
{\em Parallel spinors and connections with skew-symmetric torsion in string theory},
Asian J. Math., 6 (2002), n.2, 303--335.


\bibitem{FriKath}
Th. Friedrich and I. Kath,
{\em 7-Dimensional Compact Riemannian Manifolds with Killing Spinors},
Commun. Math. Phys. 133 (1990), 543--561.


\bibitem{Kath1}
I. Kath,
{\em Pseudo-Riemannian T-duals of compact Riemannian homogeneous spaces},
Transformation Groups Vol. 5, No.2 (2000), 157-179.


\bibitem{KobNomi}
S. Kobayashi and K. Nomizu,
{\em Foundations of Differential Geometry},
Vol. 1 and 2, Interscience Wiley, 1963 and 1969.


\bibitem{Oku}
M. Okumura,
{\em Some remarks on space with a certain contact structure},
T\^ohoku Math. J. (2) 14 (1962), 135--145.


\bibitem{WangZil}
Mc. Y. Wang and W. Ziller,
{\em On normal homogeneous Einstein manifolds},
Annales Sc. Ec. Norm. Sup. 4e s\'erie, tome 18, no 4 (1985), 563--633.







\end{thebibliography}
\end{document}